\documentclass[12pt]{article}
\usepackage[a4paper,margin=3cm]{geometry}
\usepackage{amsmath}
\usepackage{amssymb}

\newtheorem{theorem}{Theorem}[section]
\newtheorem{lemma}[theorem]{Lemma}
\newtheorem{corollary}[theorem]{Corollary}

\newenvironment{proof}{{\bf Proof.}}{\hfill$\Box$\\}
\newenvironment{definition}{{\vskip 3ex\bf Definition.} }{\\}
\newenvironment{remark}{{\vskip 3ex\bf Remark.} }{\\}

\newcommand{\R}{\mathbb{R}}

\newcommand{\ad}{\mathrm{ad}}
\newcommand{\rk}{\mathrm{rk}}

\newcommand{\spn}{\mathrm{span}}

\newcommand{\g}{\mathfrak{g}}

\title{\bf Parabolic $(3,5,6)$-distributions and $Gl(2)$-structures}
\author{Wojciech Kry\'nski\thanks{
Institute of Mathematics, Polish Academy of Sciences, ul.~\'Sniadeckich 8, 00-956 Warszawa, Poland. E-mail: krynski@impan.gov.pl.\newline
Supported by Junior Research Fellowship from Erwin Schr\"odinger Institute, Vienna.
}}

\begin{document}
\maketitle

\begin{abstract}
We consider rank 3 distributions with growth vector $(3,5,6)$. The class of such distributions splits into three subclasses: parabolic, hyperbolic and elliptic. In the present paper, we deal with the parabolic case. We provide a classification of such distributions and exhibit connections between them and $Gl(2)$-structures. We prove that any $Gl(2)$-structure on three and four dimensional manifold can be interpreted as a parabolic $(3,5,6)$-distribution.
\end{abstract}

{\bf Keywords:} $Gl(2)$-structures, vector distributions

\section{Introduction}\label{s1}
In the recent years $Gl(2)$-structures have attracted much attention due to their links to ODEs. The first result in this direction goes back to the paper of S-S.~Chern \cite{Ch} who showed that if an ordinary differential equation of third order satisfies W\"unschmann condition then it defines a conformal Lorentz metric on its solutions space. The similar observation for ODEs of 4th order was made by R.~Bryant in his paper on exotic holonomies \cite{B}. The general case was treated by M.~Dunajski and P.~Tod \cite{DT} whereas a more detailed analysis of equations of order 5 was given in \cite{GN}. The links between ODEs and $Gl(2)$-structures were also exhibited in \cite{DG,FKN,K2,N}.

Simultaneously, the serious progres has been made in understanding geometry of non-holonomic distributions and wide classes of distributions have been classified \cite{B2,DZ1,DZ2,KZ}.
It is worth to mention that the new impact came form control theory and works on so-called singular curves which, in many cases, allow to understand the geometry of distributions (see \cite{AS,JKP,K1,M,Z}).
In the present paper we show that the two topics: $Gl(2)$-structures and distributions, are strongly related.

To be more precise, we consider rank 3 distribution $D$ on a 6-dimensional manifold $M$ and assume that $D$ has growth vector $(3,5,6)$ (i.e.\ $\rk\,[D,D]=5$ and $[D,[D,D]]=TM$; we say that $D$ is (3,5,6)-distribution). The study of such a class of distributions was initiated by B.~Doubrov \cite{D}. He showed that the class splits into three subclasses distinguished by the signature of a certain bilinear form associated to a distribution, i.e.~$D$ can be either parabolic, elliptic or hyperbolic. Doubrov concentrated on elliptic and hyperbolic cases, which can be described in terms of Cartan geometry modelled on $Sl(4)$. The distributions of elliptic and hyperbolic type correspond to systems of two PDEs, elliptic or hyperbolic respectively, for one function in two independent variables. The results of Doubrov give clear and complete picture of the geometry of such PDEs.

The geometry of parabolic $(3,5,6)$-distributions is more complicated. First of all the algebra of infinitesimal symmetries can be infinite dimensional and thus, \emph{a priori}, there is no hope for Cartan connection. Moreover there is no unique symmetric model and in fact there is a splitting to a number of subclasses. There is also no clear link to the theory of PDEs since parabolic systems of two PDEs in two variables give rise to $(3,4,6)$-distributions, which can be reduced to $(2,3,5)$-distributions considered by Cartan in his famous paper \cite{C}. However, in the present paper we managed to solve the problem of classification of parabolic $(3,5,6)$-distributions almost completely (we omit one branch only). Additionally, we discover a remarkable phenomenon that all $Gl(2)$-structures on 3 and 4 dimensional manifolds can be interpreted as some parabolic $(3,5,6)$-distributions. In dimension 4 the result does not depend on the fact whether a $Gl(2)$-structure is defined by an ODE or not (in the case of dimension 3 all $Gl(2)$-structures are of equation type, see \cite{FKN,K2}). Therefore, as a by-product, we get a unified model for all $Gl(2)$-structures on 4-dimensional manifolds.
\vskip 1ex
{\bf Acknowledges.} I would like to express my gratitude to Brois Doubrov for his comments and questions posted to me while I was working on this paper.

\section{Preliminaries}\label{s2}
Let $D$ be a $(3,5,6)$-distribution on a manifold $M$. Then, the Lie bracket of vector fields gives rise, at each point $x\in M$, to the surjection $D(x)\wedge D(x)\to[D,D](x)/D(x)$
$$
(v_1,v_2)\mapsto [V_1,V_2](x)\mod D(x),
$$
where in order to compute the Lie bracket on the right hand side we extend vectors $v_1,v_2\in D(x)$ to local sections $V_1$ and $V_2$ of $D$ in an arbitrary way, and the result does not depend on an extension. Since $\dim D(x)\wedge D(x)=3$ and $\dim [D,D](x)/D(x)=2$ the mapping has one dimensional kernel. Any element of $D(x)\wedge D(x)$ is decomposable and thus there is the unique subdistribution
$$
D_2\subset D
$$
of rank 2 such that $[D_2,D_2]\subset D$ (one can consider $D_2$ as a square root of $D$). There are two cases: $D_2$ can be integrable or $[D_2,D_2]=D$. In the second case $D$ is uniquely determined by $D_2$ which has growth vector $(2,3,5,6)$. All distributions of type $(2,3,5,6)$ were classified by B.~Doubrov and I.~Zelenko in \cite{DZ1} and therefore, in what follows, we will assume that $D_2$ is integrable. For a convenience we will denote
$$
D_5=[D,D].
$$

Now, assume that $(X_1,X_2)$ is a local frame of $D_2$ and let $Y$ be a vector field complementing $X_1$ and $X_2$ to a local frame of $D$. Define
$$
Y_i=[Y,X_i].
$$
Then $(X_1,X_2,Y,Y_1,Y_2)$ is a local frame of $[D,D]$ and we can complement this tuple to the full local frame on $M$ by choosing a vector field $Z$. Following B.~Doubrov \cite{D} we define a $2\times 2$ matrix-valued function $(a_{ij})$ by the formula
$$
[X_i,Y_j](x)=a_{ij}(x)Z(x)\mod [D,D](x).
$$
It is straightforward to check that for any $x\in M$ the matrix $(a_{ij}(x))$ is symmetric and if we make a different choice of $Y$ and $Z$ then it is multiplied by a number. Moreover if we take different $X_1$ and $X_2$ the matrix $(a_{ij}(x))$ transforms as a bilinear form. Therefore at each point $x\in M$ there is a well defined bilinear symmetric form on $D_2(x)$ given up to multiplication by a number. There are three cases depending on the signature of $(a_{ij}(x))$: if $(a_{ij}(x))$ is definite then we say that $D$ is \emph{elliptic at $x$}, if $(a_{ij}(x))$ is indefinite then we say that $D$ is \emph{hyperbolic at $x$} or if $(a_{ij}(x))$ is not of a full rank then we say that $D$ is \emph{parabolic at $x$}. The parabolic case splits to the two subsequent cases: if $(a_{ij}(x))$ has rank 1 then we say that $D$ is \emph{non-degenerated parabolic at $x$} or if $(a_{ij}(x))$ has rank 0 then we say that $D$ is \emph{degenerated parabolic at $x$}.

\begin{definition}
A $(3,5,6)$-distribution $D$ is \emph{regular at $x\in M$} if there exists a neighbourhood of $x$ such that the signature of $(a_{ij})$ is constant in this neighbourhood ($D$ is either elliptic or hyperbolic or non-degenerated parabolic or degenerated parabolic). Otherwise we say that $D$ is \emph{singular at $x$}.
\end{definition}

Clearly if $D$ is elliptic or hyperbolic at $x$ then it is also elliptic or hyperbolic in a small neighbourhood of $x$ and thus all elliptic and hyperbolic points are regular. On the other hand there are singular parabolic points, but we will not consider them in the present paper. We will consider a problem of local equivalence of $(3,5,6)$-distributions at regular parabolic points and thus we will just say that $D$ is \emph{parabolic} (\emph{degenerated} or \emph{non-degenerated}).

In the whole paper we say that two structures on a manifold are \emph{equivalent} if there exists a diffeomorphism transforming one structure onto the other.

\section{Reduction}\label{s3}
Assume first that $D$ is degenerated parabolic. It follows that $[X_i,Y_j]=0\mod D_5$. However, $[D,D_5]=TM$ and thus for each choice of $Y$ as in Section \ref{s2} we have a surjection $D_5(x)\to T_xM/D_5(x)$
$$
v\mapsto [Y,V](x)\mod D_5(x)
$$
where $V$ is an extension of $v\in D_5(x)$ to a local section of $D_5$. The mapping has 4-dimensional kernel which does not depend on the choice of $Y$. Therefore there is a well defined subdistribution $D_4\subset D_5$. Clearly $D\subset D_4$. Now we can consider the mapping $D_2(x)\to D_5(x)/D_4(x)$
$$
v\mapsto [Y,V](x)\mod D_4(x)
$$
where $V$ is an extension of $v\in D_2(x)$ to a local section of $D_2$. This mapping has one dimensional kernel $D_1\subset D_2$ which is again invariantly assigned to a distribution.

If $D$ is non-degenerated parabolic then the situation looks similar. Namely the matrix $(a_{ij}(x))$ which is a bilinear form on $D_2$ has one dimensional kernel for any $x$ and thus we have rank one distribution $D_1\subset D_2$. Then we can define $D_4=[D_1,D]$.

Denoting $D_3=D$, in both cases we get the flag
$$
D_1\subset D_2\subset D_3\subset D_4\subset D_5\subset TM.
$$
\begin{lemma}\label{l1}
If $D$ is a regular parabolic $(3,5,6)$-distribution on a manifold $M$ then the associated flag $(D_i)_{i=1,\ldots,5}$ satisfies
\begin{equation}\label{eq1}
[D_1,D_2]=D_2,\qquad[D_1,D_3]=D_4,\qquad [D_1,D_4]=D_4
\end{equation}
\end{lemma}
\begin{proof}
The first equality follows from the fact that $D_2$ is always integrable in our considerations. The second follows from the definition of $D_1$ and $D_4$. In order to prove $[D_1,D_4]=D_4$ let us assume that $(X_1,X_2,Y,Y_1,Y_2,Z)$ is a local frame on $M$ as in Section \ref{s2}. Moreover assume that $X_1$ spans $D_1$. Then $Y_1=[Y,X_1]$ complements $(X_1,X_2,Y)$ to a local frame of $D_4$. We shall show that $[X_1,Y_1]=0\mod D_4$. In general we have
$$
[X_1,Y_1]=fY_2\mod D_4
$$
for some $f$. The proof splits into two cases.

In the degenerated case we may assume that $Z=[Y,Y_2]$. Then we consider $[Y,[X_1,Y_1]]$ and apply Jacobi identity. On the one hand we get
$$
[Y,[X_1,Y_1]]=fZ\mod D_5
$$
On the other hand we get
$$
[Y,[X_1,Y_1]]=[Y_1,Y_1]+[X_1,[Y,Y_1]]=0+[X_1,\tilde Y]\mod D_5
$$
for some section $\tilde Y$ of $D_5$. But, since $(a_{ij})=0$, we get $[X_1,\tilde Y]=0\mod D_5$ and consequently $f=0$.

In the non-degenerate case we may assume that $Z=[X_2,Y_2]$. Then we consider $[X_2,[X_1,Y_1]]$ and apply Jacobi identity. On the one hand we get
$$
[X_2,[X_1,Y_1]]=fZ\mod D_5
$$
On the other hand we get
$$
[X_2,[X_1,Y_1]]=[\tilde X,Y_1]+[X_1,\tilde Y]\mod D_5
$$
for some section $\tilde X$ of $D_2$ and some section $\tilde Y$ of $D_5$. But, by definition of our frame the only non-zero entry of $(a_{ij})$ is $a_{22}$ and thus we get $[\tilde X,Y_1]=[X_1,\tilde Y]=0\mod D_5$. As a result we get $f=0$ as desired.
\end{proof}

Now we can define the fundamental \emph{reduction} of parabolic $(3,5,6)$-distribution. Namely, from \eqref{eq1} it follows that $D_1$ is contained in Cauchy characteristic of both $D_2$ and $D_4$. Thus we can consider (at least locally) the quotient manifold $N=M/D_1$ with the quotient mapping
$$
q\colon M\to N
$$
and with two well defined distributions
$$
B_1=q_*(D_2),\qquad B_3=q_*(D_4)
$$
such that $\rk\,B_1=1$, $\rk\,B_3=3$ and $B_1\subset B_3$. A  pair $(B_1,B_3)$ on $N$ will be called the \emph{reduced pair of $D$}.

There exists also a converse construction and it appears that the reduced pair $(B_1,B_3)$ contains all information about the original distribution $D$. Having a pair $(B_1,B_3)$ on a manifold $N$ we consider first the quotient vector bundle  $B_3/B_1\to N$ and then we define a manifold
$$
\tilde M=P(B_3/B_1)
$$
by taking the total space of the projectivisation of the bundle $B_3/B_1\to N$. $\tilde M$ is a manifold of dimension $\dim N+1$ and we have a fibration $\pi\colon\tilde M\to N$. On $\tilde M$ we define a canonical rank 3 distribution
$$
\tilde D_3(x)=\{v\in T_x\tilde M\ |\ \pi_*(v)\in L(x)\}
$$
where for $x\in\tilde M$, which is an element of $P(B_3(\pi(x))/B_1(\pi(x)))$, we denote by $L(x)$ a two dimensional subspace of $B_3(\pi(x))$ containing $B_1(\pi(x))$ and defining $x\in P(B_3(\pi(x))/B_1(\pi(x)))$. By definition $\tilde D_3$ contains the vertical rank-one distribution $\tilde D_1$ tangent to the fibres of $\pi$. There is also a well defined subdistribution $\tilde D_2=\pi_*^{-1}(B_1)$ of rank 2. Moreover, it can be easily seen that $\tilde D_4=[\tilde D_1,\tilde D_3]$ is a rank 4 distribution which coincide with $\pi_*^{-1}(B_3)$. It follows that the flag $\tilde D_1\subset\tilde D_2\subset\tilde D_3\subset\tilde D_4$ satisfies relations \eqref{eq1}.
\begin{lemma}\label{l2}
Assume that $D_1\subset D_2\subset D_3\subset D_4$ satisfies \eqref{eq1}. Then the natural mapping $\Phi\colon M\to \tilde M$
$$
\Phi(x)=q_*(D_3(x))/B_1(q(x))\in P(B_3(q(x))/B_1(q(x)))
$$
is a local diffeomorphism which establishes an equivalence of flags $(D_i)_{i=1,\ldots,4}$ and $(\tilde D_i)_{i=1,\ldots,4}$. In particular $D_3$ and $\tilde D_3$ are equivalent.
\end{lemma}
\begin{proof}
The first and the third relation of \eqref{eq1} allow us to define the reduction. By construction $\pi_*(\Phi_*(D_i(x)))=\pi_*(\tilde D_i(\Phi(x)))=q_*(D_i(x))$ for any $x\in M$ and $i=1,2,3,4$. Therefore, in order to finish the proof it is sufficient to prove that $\Phi$ is a local diffeomorphism, i.e.\ the fibres of $q$ are transformed onto the fibres of $\pi$. But this follows from the second relation of \eqref{eq1}. Namely there exists a section $X$ of $D_1$ and a section $Y$ of $D_3$, transversal to $D_2\subset D_3$ such that $[X,Y]$ is transversal to $D_3\subset D_4$. It follows that $\Phi_*(X)\neq 0$.
\end{proof}

Directly from Lemma \ref{l2} we get the following
\begin{corollary}\label{c1}
Two regular parabolic $(3,5,6)$-distributions $D$ and $D'$ are equivalent if and only if the corresponding reduced pairs $(B_1,B_3)$ and $(B_1',B_2')$ are equivalent.
\end{corollary}

\section{Degenerated case}\label{s4}
In this section we consider degenerated parabolic $(3,5,6)$-distributions.
\begin{lemma}\label{l3}
Let $D$ be a degenerated parabolic $(3,5,6)$-distribution on $M$ and let $(B_1,B_3)$ be the associated reduced pair on $N$. Then
\begin{enumerate}
\item $B_4=[B_3,B_3]=[B_1,B_3]$ is a rank 4 distribution,
\item $[B_1,B_4]=B_4$, $[B_4,B_4]=TN$,
\item $B_3$ has a rank one Cauchy characteristic $C$ and $B_2=C\oplus B_1$ is Cauchy characteristic of $B_4$ (in particular it is integrable).
\end{enumerate}
\end{lemma}
\begin{proof}
The first two statements follows immediately from the definition of the flag $(D_i)_{i=1,\ldots,5}$ and the definition of $B_i$. Namely, in the degenerated case, $[D_4,D_4]=[D_2,D_4]=D_5$, $[D_2,D_5]=D_5$ and $[D_5,D_5]=TM$.

It follows from statements 1 and 2 that $B_3$ has growth vector $(3,4,5)$ and thus it has Cauchy characteristic $C$ which is of rank one. $C$ does not coincide with $B_1$ since $[B_1,B_3]=B_4$. Therefore $B_2=C\oplus B_1$ is a distribution of rank 2. To prove that $B_2$ is Cauchy characteristic of $B_4$ let us choose a vector field $V$ which spans $B_1$ and a vector field $W$ which spans $C$. Moreover, let $U$ be a vector field complementing $(V,W)$ to a local frame of $B_3$. Then $(V,W,U,[V,U])$ is a local frame of $B_4$. Additionally $[W,V]=[W,U]=0\mod B_3$, since $W$ is a section of Cauchy characteristic of $B_3$. Hence
$$
[W,[V,U]]=[[W,V],U]+[V,[W,U]]=0\mod B_4.
$$
\end{proof}

Now we are ready to prove our first main result.

\begin{theorem}\label{t1}
All degenerated parabolic $(3,5,6)$-distributions are locally equivalent to the canonical Cartan distribution on the mixed jet space $J^{2,1}(\R,\R^2)$. In natural coordinates $(t,u,v,u_1,u_2,v_1)$ on $J^{2,1}(\R,\R^2)$ the distribution is annihilated by the following one-forms
$$
du-u_1dt,\qquad du_1-u_2dt,\qquad dv-v_1dt.
$$
\end{theorem}
\begin{proof}
For a given degenerated parabolic $D$ on a manifold $M$ we have defined in Lemma \ref{l3} the flag $B_1\subset B_2\subset B_3\subset B_4\subset TN$ on a manifold $N$ such that $B_1$ is contained in Cauchy characteristics of $B_2$ and $B_4$. Moreover $[B_2,B_3]=B_4$. The situation on $N$ is completely similar to the situation on $M$ described in the previous section. We can define new manifold $O=N/B_1$ with the quotient mapping $p\colon N\to O$ and with two distributions $A_1=p_*(B_2)$ and $A_3=p_*(B_4)$ of rank 1 and 3, respectively. Making a repetition of the reasoning of previous section we get that the pair $(A_1,A_3)$ completely determines $B_3$. It also determines $B_1$ which can be recover as a distribution tangent to the fibres of $p$. Thus $(A_1,A_3)$ completely determines the original distribution $D$.

Now, $A_3$ is a contact distribution on 4-dimensional manifold and $A_1$ is its characteristic subdistribution. By Darboux theorem, all such pairs $(A_1,A_3)$ are equivalent. Therefore all degenerated parabolic $(3,5,6)$-distributions are locally equivalent. Cartan distribution on $J^{2,1}(\R,\R^2)$ is degenerated parabolic, and hence it can be taken as a model.
\end{proof}

\section{Non-degenerated case}\label{s5}
In this section we consider non-degenerated parabolic $(3,5,6)$-distributions. In this case, besides relations \eqref{eq1}, we also have the following relations
\begin{equation}\label{eq2}
[D_1,D_5]=D_5,\quad[D_2,D_3]=D_5,\quad[D_2,D_4]=D_5, \quad[D_2,D_5]=TM.
\end{equation}
All of them follows directly form the definitions. It is reasonably to introduce the following graded Lie algebra at each point $x\in M$
$$
\g(x)=\bigoplus_{i=1}^7\g_{-i}(x)
$$
where
$$
\begin{array}{lll}
\g_{-1}(x)=D_1(x),\quad& \g_{-2}(x)=D_2(x)/D_1(x),\quad& \g_{-3}(x)=D_3(x)/D_2(x),\\
\g_{-4}(x)=D_4(x)/D_3(x),\quad& \g_{-5}(x)=D_5(x)/D_4(x),\quad& \g_{-6}(x)=0,\\ &\g_{-7}(x)=T_xM/D_5(x).&\\
\end{array}
$$
and bracket in $\g(x)$ is defined in a standard way using Lie bracket of vector fields. Lie algebra $\g(x)$ is assigned to a distribution $D$ at point $x$ in an invariant way and it will be called \emph{symbol algebra}

Our first aim is to classified all possible graded Lie algebras $\g=\bigoplus_{i=1}^7\g_{-i}$ which can appear as a symbol of a non-degenerated parabolic $(3,5,6)$-distribution. 
\begin{lemma}\label{l4}
Let $\g=\bigoplus_{i=1}^7\g_{-i}$ be a symbol algebra of a non-degenerated parabolic $(3,5,6)$-distribution $D$ at some point $x\in M$. Then there exists a basis $e_1,\ldots,e_5,e_7$ of $\g$ such that $e_i$ spans $\g_{-i}$,
\begin{enumerate}
\item $[e_1,e_2]=0$,
\item $[e_1,e_3]=e_4$,
\item $[e_1,e_4]=0$,
\item $[e_2,e_3]=e_5$,
\item $[e_2,e_5]=e_7$,
\item $[e_3,e_4]=de_7$,
\end{enumerate}
where $d=0$ or $d=1$ and all other brackets $[e_i,e_j]$ vanish. Thus, there are exactly two non-equivalent symbols.
\end{lemma}
\begin{proof}
First of all, Lie brackets $[e_i,e_j]$ not listed in the lemma necessarily vanish due to the definition of graded Lie algebra.

If we choose a local frame $(X_1,X_2,Y,Y_1,Y_2,Z)$ as in Section \ref{s2} in such a way that $X_1$ spans $D_1$ then we can take $e_1=X_1(x)$, $e_2=X_2(x)$, $e_3=Y(x)$, $e_4=Y_1(x)$, $e_5=Y_2(x)$ and $e_7=Z(x)$. Then we have $[e_1,e_3]=e_4$ and $[e_2,e_3]=e_5$. Moreover, we can assume that $Z=[X_2,Y_2]$ and then $[e_2,e_5]=e_7$. Lie bracket $[e_1,e_2]$ vanish because $D_2$ is integrable. Lie bracket $[e_1,e_4]$ vanish due to the third relation of Lemma \ref{l1}. Then we have $[e_3,e_4]=de_7$ for some $d\in\R$. However, if we substitute $Y:=\alpha Y$ for some $\alpha\neq 0$ then a simple calculation proves that $d$ becomes $\alpha d$. Hence, if $d\neq 0$ we can assume $d=1$.
\end{proof}

\begin{remark}
The two symbols appear in the paper \cite{R} and are denoted: m6\_3\_3 (for $d=0$) and m6\_3\_4 (for $d=1$). Flat distributions on Lie groups corresponding to thees two graded Lie algebras have infinite dimensional algebras of infinitesimal symmetries. For simplicity, m6\_3\_3 we will denote $\g^0$ and m6\_3\_4 we will denote $\g^1$.
\end{remark}

In the next two lemmas we provide basic properties of the reduced pair $(B_1,B_3)$ associated to a non-degenerated parabolic distribution.

\begin{lemma}\label{l6}
Let $D$ be a non-degenerated parabolic $(3,5,6)$-distribution on $M$ and let $(B_1,B_3)$ be the associated reduced pair on $N$. Then
\begin{enumerate}
\item $[B_1,B_3]=B_4$ is a rank 4 distribution,
\item $[B_1,B_4]=TN$,
\item there exists a unique rank 2 subdistribution $B_2\subset B_3$ such that $[B_2,B_2]\subset B_3$ and $B_1\subset B_2$.
\end{enumerate}
Conversely, if a pair $(B_1,B_3)$ satisfies conditions 1 and 2 above, then $\tilde D$ defined in Section \ref{s3} on manifold $\tilde M$ is a non-degenerated parabolic $(3,5,6)$-distribution.
\end{lemma}
\begin{proof}
The first two statements immediately follows from the definition of the flag $(D_i)_{i=1,\ldots,5}$ and the definition of $B_i$. Namely, in the non-degenerated case, $[D_2,D_4]=D_5$ and $[D_2,D_5]=TM$ (see formula \eqref{eq2}).

To prove statement 3 let us choose a vector field $X$ which spans $B_1$ and consider a mapping $B_3(x)\to B_4(x)/B_3(x)$ defined by the formula $v\mapsto [X,V](x)\mod B_3(x)$, where $V$ is an extension of $v\in B_3(x)$ to a local section of $B_3$. It follows from statement 1 that this mapping has two dimensional kernel and in this way we define $B_2(x)$.

Note that if $B_3$ has growth vector $(3,4,5)$ then it has Cauchy characteristic $C$, which is a distribution of rank 1. Then $B_2=B_1\oplus C$. If $B_3$ has growth vector $(3,5)$ then it is well known that the square root of $B_3$ exists. This square root is exactly $B_2$ defined above.

To prove that any pair $(B_1,B_2)$ satisfying conditions 1 and 2 defines a non-degenerated parabolic $(3,5,6)$-distribution $\tilde D$ on $\tilde M$ it is sufficient to show that $\rk\,[\tilde D,\tilde D]=5$ or $\pi_*([\tilde D,\tilde D])=B_4$. Assume that vector fields $X,Y,Z$ are given, such that $X$ spans $B_1$, $(X,Y)$ is a frame of $B_2$ and $(X,Y,Z)$ is a frame of $B_3$. Additionally, if we assume that $s$ is a parameter of the fiber of $\pi\colon\tilde M\to N$ then we can write $\tilde D=\spn\{\partial_s, X,Z+sY\}$. It follows that $[\tilde D,\tilde D]=\spn\{\partial_s,X,Y,Z,[X,Z]\}=\spn\{\partial_s\}\oplus B_4$ as desired.
\end{proof}

\begin{lemma}\label{l5}
Let $D$ be a non-degenerated parabolic $(3,5,6)$-distribution on $M$, let $(B_1,B_3)$ be the associated reduced pair and let $\g(x)$ be a symbol algebra of $D$ at $x\in M$. $\g(x)$ is isomporphic to $\g^0$ iff $B^3$ has growth vector $(3,4,5)$ at $x$. $\g(x)$ is isomporphic to $\g^1$ iff $B^3$ at $x$ has growth vector $(3,5)$ at $x$.
\end{lemma}
\begin{proof}
By Lemma \ref{l6} there are only two possible growth vectors of $B_3$: $(3,4,5)$ or $(3,5)$ (because $\rk\,[B_3,B_3]\geq\rk\,[B_1,B_3]=4$ and $[B_3,[B_3,B_3]]\supset[B_1,B_4]=TN$). Let us choose a local frame $(X_1,X_2,Y,Y_1,Y_2,Z)$ as in the proof of Lemma \ref{l4}. If we take into account that $B_3=q_*(D_4)=q_*(\spn\{X_1,X_2,Y,Y_1\})$ then a characterisation of symbol $\g(x)$ in terms of $B_3$ becomes obvious.
\end{proof}

In order to exclude singular points we will need one more regularity condition.

\begin{definition}
A non-degenerated parabolic $(3,5,6)$-distribution $D$ is called \emph{completely non-degenerated} if the associated distribution $B_2$ has locally constant growth vector.
\end{definition}

It follows that if $D$ is completely non-degenerated then either $B_2$ is integrable or $[B_2,B_2]=B_3$. In the second case $B_3$ is determined by $B_2$ and thus  $D$ is determined by the pair $(B_1,B_2)$.

In view of Lemma \ref{l5} and Lemma \ref{l6} there are four possibilities at a point $x\in M$. A non-degenerated parabolic $(3,5,6)$-distribution can have a symbol algebra $\g^0$ or $\g^1$ and $B_2$ can be integrable or not. 
Note that if $D$ has symbol $\g^1$ at $x$ then it has symbol $\g^1$ in a neighbourhood of $x$.

\begin{remark}
In the paper \cite{K2} we have introduced the notion of \emph{regular pairs}. We say that a pair $(E,F)$ of two distributions on a manifold $M$ of dimension $n$ is regular if
\begin{enumerate}
\item $\rk\,E=1$, $\rk\, F=2$, $E\subset F$
\item $\rk\,\ad_E^iF=i+2$ for $i=1,\ldots,n-2$, where $\ad^i_EF$ are distributions defined by induction: $\ad_EF=[E,F]$ and $\ad^{i+1}_EF=[E,\ad^i_EF]$.
\end{enumerate}

It is proved in \cite{K2} that the notion of regular pairs generalises the notion of ODEs. Namely, for a given equation of order $k+1$ we have a canonical regular pair on the space of $k$-jets. The pair consists of Cartan distribution and rank one distribution spanned by the total derivative. Pairs which are locally diffeomorphic to pairs which come from ODEs are called \emph{of equation type}. An intrinsic characterisation of such pairs is given in \cite{K2}.

For a regular pair there is also a notion of \emph{W\"unschman condition}, which generalises the notion of W\"unschman condition in the case of ODEs. We proved in \cite{K2} that there is one-to-one correspondence between $Gl(2)$-structures and regular pairs satisfying W\"unschman condition. 
\end{remark}

Now we are in position to state our main results.

\begin{theorem}\label{t2}
Let $D$ be a completely non-degenerated parabolic $(3,5,6)$-distribution on $M$ such that the associated distribution $B_2$ on $N$ is non-integrable. Then
\begin{enumerate}
\item the pair $(B_1,B_2)$ is regular in the sense of \cite{K2},
\item there exists a canonical frame on a $T(2)$-bundle over $N$ and two distributions are equivalent iff the corresponding frames are diffeomorphic (here $T(2)\subset Gl(2)$ is the subgroup of upper-triangular matrices),
\item the pair $(B_1,B_2)$ is of equation type iff $D$ has constant symbol algebra $\g^0$,
\item if $(B_1,B_2)$ satisfies W\"unschmann condition then it defines a $Gl(2)$-structure on the quotient manifold $N/B_1$, which is of dimension 4; conversely all germs of $Gl(2)$-structures on 4-dimensional manifolds can be obtained in this way.
\end{enumerate}
\end{theorem}
\begin{proof}
Statement 1 follows from the assumption: $[B_2,B_2]=B_3$, which implies $[B_1,B_2]=B_3$ and together with Lemma \ref{l6} proves that $(B_1,B_2)$ is a regular pair. In order to prove statement 3 let us recall from \cite{K2} that a regular pair $(E,F)$ on 5-dimensional manifold is of equation type if and only if $[F,F]$ has growth vector $(3,4,5)$. Thus statement 3 follows from Lemma \ref{l5}. Statement 4 is just a consequence of statement 1 and results of \cite{K2} (Theorem 1.1).

Statement 2 in the case of regular pairs of equation type follows from \cite{DKM}, where a prove is given that for an arbitrary equation of 4th order there is a normal Cartan connection on a $T(2)$-bundle. The result can be generalised to the case of an arbitrary regular pair and we will provide a proof in a forthcoming paper \cite{K3}. However, in the general case one do not get a Cartan connection, but just a frame on a bundle.
\end{proof}

\begin{theorem}\label{t3}
Let $D$ be a completely non-degenerated parabolic $(3,5,6)$-distribution on $M$ such that the associated distribution $B_2$ on $N$ is integrable. If $D$ has constant symbol $\g^0$ then
\begin{enumerate}
\item Cauchy characteristic $C$ of $B_3$ is contained in Cauchy caracteristic of $B_2$ and thus there is a well defined reduction $p\colon N\to O$, $A_1=p_*(B_2)$, $A_2=p_*(B_3)$, where $O=N/C$, and $D$ is uniquely defined by the pair $(A_1, A_2)$
\item the pair $(A_1,A_2)$  is regular in the sense of \cite{K2},
\item if $(A_1,A_2)$ satisfies W\"unschmann condition then it defines a conformal Lorentz metric on the quotient manifold $O/A_1$, which is of dimension 3; conversely all germs of conformal Lorentz metrics on 3-dimensional manifolds can be obtained in this way.
\end{enumerate}
\end{theorem}
\begin{proof}
Statement 1 follows from the fact that $B_2=B_1\oplus C$, where $C$ is Cauchy characteristic of $B_3$ (see proof of Lemma \ref{l6}). If $B_2$ is integrable then $C$ is also, obviously, contained in Cauchy characteristic of $B_2$ and thus the reduction is well defined. Statement 2 follows from the fact that $B_3$ has growth vector $(3,4,5)$ which implies that $A_2$ has growth vector $(2,3,4)$. Additionally $[A_1,A_3]=TO$ since $[B_2,B_4]=TN$ as was proved in Lemma \ref{l6}. Statement 3 is a corollary of Theorem 1.1 \cite{K2} applied to regular pairs on 4-dimensional manifolds.
\end{proof}
\begin{remark}
Since all regular pairs on 4-dimensional manifolds are of equation type (see \cite{FKN,K2}), the problem of equivalence of parabolic $(3,5,6)$-distributions described in Theorem \ref{t3} is reduced to the problem of contact equivalence of ODEs of third order. The last problem was solved by Chern \cite{Ch} who constructed a Cartan connection taking values in $\mathfrak{sp}(4,\R)$ (see also \cite{DKM}).
\end{remark}
\vskip 2ex
{\bf Open problem.} Classify all non-degenerated parabolic $(3,5,6)$-distributions with integrable $B_2$ and symbol $\g^1$. In this case $B_3$ is equivalent to Cartan distribution on the space $J^1(\R, \R^2)$ and $B_2$ is its integrable subdistribution tangent to the fibres of the projection $J^1(\R, \R^2)\to J^0(\R, \R^2)$. However the choice of $B_1\subset B_2$ seems to lead to non-equivalent $D$.

\section{Symmetric models and PDEs}
In this section we will provide examples of non-degenerated parabolic $(3,5,6)$-distributions.

We start with two flat models. Namely, for algebras $\g^0$ and $\g^1$ we can construct Lie groups $G^0$ and $G^1$ such that $\g^i$ is the Lie algebra of $G^i$. Then on $G^i$, for $i=0,1$, we can define a left-invariant rank 3 distribution $D^i$ such that at the identity element $e\in G^i$ we have $D^i(e)=\g^i_{-1}\oplus\g^i_{-2}\oplus\g^i_{-3}$. Then it is clear that on $G^i$ there exists a frame $(X_1,X_2,Y,Y_1,Y_2,Z)$ of left invariant vector fields which is adapted to $D^i$ in a sense of the proof of Lemma \ref{l4} and have structural constants such as algebra $\g^i$. Moreover any distribution which has an adapted frame with structural constants such as algebra $\g^i$ is locally equivalent to the distribution $D^i$ on $G^i$. We call $D^i$ the \emph{flat distribution} of type $\g^i$. Below we will presents PDE models for $D^0$ and $D^1$, but before we do this we will show models of distributions corresponding to ODEs from Theorems \ref{t2} and \ref{t3}. They can be relatively easy obtained in a proces inverse to the reduction of Section \ref{s3}.

Any non-degenerated parabolic $(3,5,6)$-distribution $D$ with integrable $B_2$ and constant symbol algebra $\g^0$ is locally equivalent to a distribution on $\R^6$, with coordinates $(u_1,u_2,u_3,x,y,z)$, annihilated by the following one-forms
$$
du_1-u_2dx,\qquad du_2-zdx,\qquad du_3+F(x,u_1,u_2,z)ydx+zdy
$$
for a function $F$ in four variables (the function can be arbitrary). Substituting $u:=u_1$ and $v:=u_3$ we get that the distribution defines the following system of PDEs
\begin{equation}\label{eq3}
u_y=0,\qquad v_y=-u_{xx},\qquad v_x=-Fy.
\end{equation}
A function $F$ is just a function which defines the corresponding 3rd order ODE
$$
\varphi'''=F(t,\varphi,\varphi',\varphi'').
$$

Similarly, any non-degenerated parabolic $(3,5,6)$-distribution $D$ with non-integrable $B_2$ and constant symbol algebra $\g^0$ is locally equivalent to a distribution on $\R^6$, with coordinates $(u_1,u_2,u_3,x,y,z)$, annihilated by the following one-forms
$$
du_1-u_2dx,\qquad du_2-zdx,\qquad du_3-(F(x,u_1,u_2,z,u_3+yz)-yu_3-y^2z)dx+zdy
$$
for a function $F$ in five variables. Substituting $u:=u_1$ and $v:=u_3$ we get that the distribution defines the following system of PDEs
\begin{equation}\label{eq4}
u_y=0,\qquad v_y=-u_{xx},\qquad v_x=G(x,y,u,u_x,u_{xx},v)
\end{equation}
where $G(x,y,u,u_x,u_{xx},v)=F(x,u,u_x,u_{xx},v+yu_{xx})-yv-y^2u_{xx}$. A function $F$ defines the corresponding 4th order ODE
$$
\varphi^{(4)}=F(t,\varphi,\varphi',\varphi'',\varphi''').
$$

Taking $F=0$ in \eqref{eq3} and \eqref{eq4} we get the models of distributions corresponding to the trivial equations $\varphi'''=0$ and $\varphi^{(4)}=0$. Explicite, we have
\begin{equation}\label{eq5}
u_y=0,\qquad v_y=-u_{xx},\qquad v_x=0
\end{equation}
for order 3, and
\begin{equation}\label{eq6}
u_y=0,\qquad v_y=-u_{xx},\qquad v_x=-yv-y^2u_{xx}
\end{equation}
for order 4. Note that \eqref{eq5} gives also a PDE model for the flat distribution with symbol algebra $\g^0$. On the other hand \eqref{eq6} corresponds to the flat $Gl(2)$-structure on 4-dimensional manifold.

A PDE system corresponding to the flat distribution with symbol algebra $\g^1$ has the following form
$$
u_y=\frac{1}{2}(u_{xx})^2,\qquad u_{xy}=0,\qquad v_y=u_{xx},\qquad v_x=0.
$$
\vskip 2ex
{\bf Open problem.} Find PDE systems corresponding to all parabolic $(3,5,6)$-distributions with symbol algebra $\g^1$.

\end{document}